\documentclass[12pt,reqno]{amsart}

\usepackage{amsmath, amsthm, amssymb}

\usepackage{url} 
\usepackage{graphicx}
\usepackage{xcolor}
\usepackage{moreverb}
\usepackage{fancyvrb}
\usepackage{tikz}
\usetikzlibrary{arrows.meta}

\def\r{\mathbb R}

\def\s{\mathbb S}
\def\h{\mathbb H}

\def\v{\mathcal V}
\def\n{\mathcal N}
\def\t{\mathcal T}

\def\x{\mathfrak X}
\DeclareMathOperator{\sech}{sech}

\usepackage{listings}
 \lstnewenvironment{code}[1][]
  {\lstset{#1}}
  {}

\newtheorem{theorem}{Theorem}[section]
 \newtheorem{proposition}[theorem]{Proposition}
 
 \newtheorem{lemma}[theorem]{Lemma}
\theoremstyle{definition}
\newtheorem{definition}[theorem]{Definition}
\newtheorem{example}[theorem]{Example}
\newtheorem{remark}[theorem]{Remark}

\begin{document}

\title[Prescribed angle surfaces]{Prescribed Angle Surfaces Associated with Torse-Forming Vector Fields in Riemannian Manifolds}

\author{Muhittin Evren Aydin$^1$}
\address{$^1$Department of Mathematics, Faculty of Science, Firat University, Elazığ,  23200 Türkiye
\newline
ORCID: 0000-0001-9337-8165}
\email{meaydin@firat.edu.tr}
\author{ Esra Dilmen$^2$}
 \address{$^2$Firat University, Graduate School of Natural and Applied Sciences, Mathematics,
Elazığ, Türkiye
\newline
ORCID: 0000-0001-6750-9475}

 \email{esradlmn23@gmail.com}
\author{ Büşra Karakaya$^3$}
 \address{$^3$Firat University, Graduate School of Natural and Applied Sciences, Mathematics,
Elazığ, Türkiye
 \newline
ORCID: 0009-0000-6285-3701}
 \email{busrakarakaya404@gmail.com}

\begin{abstract}
In this paper, we introduce the notion of a prescribed angle hypersurface in a Riemannian manifold associated with a pair $(\mathcal{V},\theta)$, where $\mathcal{V}$ is a unit vector field along the hypersurface and $\theta$ denotes the angle between $\mathcal{V}$ and the unit normal vector field of the hypersurface. We study such hypersurfaces in case $\mathcal{V}$ is a torse-forming vector field. In the particular $3$-dimensional case, we determine the intrinsic and extrinsic curvatures of these hypersurfaces in terms of the prescribed angle and the potential function of $\mathcal{V}$. Using this, we classify prescribed angle surfaces under suitable assumptions.
\end{abstract}
\maketitle

\section{Introduction} \label{intro}

In this paper, we consider orientable surfaces in Riemannian manifolds making a prescribed angle with a vector field. Although this is a classical topic in differential geometry, we introduce a new perspective, as described below.

Let $\Sigma$ be an orientable surface in a Riemannian manifold $M$ and $\n$ denote its unit normal vector field. Let $\v$ be a nowhere-vanishing unit vector field along $\Sigma$. By a vector field along the surface $\Sigma$, we mean a section of the restricted tangent bundle $TM|_{\Sigma}$. Such vector fields may be tangent, normal or have both nonzero tangential and normal components with respect to $\Sigma$. 

The angle function $\theta$ between $\n$ and $\v$ along $\Sigma$ is defined by $\cos\theta=\langle \n,\v\rangle$. The surface $\Sigma$ is called a constant angle surface with respect to $\v$ if the function $\theta$ is constant along $\Sigma$. In this case, the vector field $\v$ is sometimes called the axis of $\Sigma$.

Such surfaces in the Euclidean space $\r^3$ were studied in \cite{cd}, where $\v$ is considered as a direction field on $\r^3$. When the ambient space is a product manifold such as $\s^2\times \r$ \cite{dfvv}, $\h^2\times \r$ \cite{dm} or $I\times_\rho \r^2$ \cite{dmvv}, constant angle surfaces were classified. In each case, the vector field $\v$ is assumed to be tangent to the $\r$-direction. If $\v$ is the position vector field on $\r^3$, then the class of constant slope surfaces was determined in \cite{mun}. 

When $\v$ is a Killing vector field, constant angle surfaces were also studied. For instance, the corresponding classifications in $\r^3$ were obtained in \cite{mun1}. In the $3$-dimensional Berger sphere, these surfaces were characterized in \cite{mo}. See also \cite{mun2} for the case where the ambient space is $\s^2\times \r$.

In higher-dimensional cases, a different perspective on the constant angle property is to require that the angle between the tangent space of a submanifold and a given vector field be constant. In this setting, and in the Euclidean case, such submanifolds are also called helix submanifolds \cite{dsr}. Helix hypersurfaces were also described in the cited article by assuming $\v$ to be a fixed direction. A recent characterization of helix surfaces was obtained in \cite{loy6} so that isogonal lines become generalized helices and pseudogeodesic lines.

Another perspective was introduced by Chen \cite{crs0}, defined rectifying submanifolds in Euclidean spaces as those for which the position vector field is orthogonal to the first normal space. Afterwards, Chen also generalized this notion to arbitrary Riemannian manifolds \cite{crs3}.

Before proceeding with the literature review, we recall the notion of a torse-forming vector field. Let $\nabla$ denote the Levi-Civita connection on a Riemannian manifold $M$. A vector field $\v\in \x(M)$ is said to be {\it torse-forming} \cite{ya0} if there exist a smooth function $f$ on $M$ and a $1$-form $\omega$ such that
\begin{equation}\label{int-tor}
\nabla_X\v=fX+\omega(X)\v, \quad X\in \mathfrak{X}(M),
\end{equation}
where $f$ and $\omega$ are called the potential function and the generating form of $\v$, respectively (see also \cite{bo1,mih0}).  In particular, the vector field $\v$ is said to be concircular if $\omega$ vanishes identically \cite{ya1}. Otherwise, that is, when $\omega\neq 0$, it is called a torqued vector field provided that $\omega(\v)=0$. Moreover, if $\omega=-f\nu$, where $\nu$ is the $1$-form dual to $\v$, then the vector field $\v$ is said to be anti-torqued \cite{cs,dan,na}.

These vector fields provide a useful tool for studying the geometry and topology of ambient spaces. However, not every Riemannian manifold admits such vector fields globally. For results concerning the global existence of these vector fields on the ambient space, see \cite{amc,crs00,c01,dan,dta}. 

Continuing the literature review on constant angle hypersurfaces, a new approach to their study was initiated in \cite{loy4,loy5}. Let the ambient space be one of the space forms $\r^n$, $\h^n$ or $\s^n$. In this setting, a concircular hypersurface $\Sigma$ is defined as a hypersurface for which $\langle \n,\v\rangle$ is equal to a constant $\theta_0$ on $\Sigma$, where $\v$ is a globally defined concircular vector field on the ambient space. Following this idea, the first author, together with Mihai and Özgür \cite{amc2}, defined and characterized anti-torqued and torqued hypersurfaces in arbitrary Riemannian manifolds.

Notice that if $\v$ is a globally defined concircular vector field, then for a concircular hypersurface the angle function $\theta$ between $\n$ and $\v$ is given by $\cos\theta=\theta_0|\v|^{-1}$. This angle function is constant if and only if $|\v|$ is constant. In this case, the hypersurface is totally umbilical (see Proposition \ref{p-umb} and \cite[Proposition 4]{loy5}). Otherwise, the angle function is non-constant along the hypersurface. This situation also occurs for torqued hypersurfaces. However, since every anti-torqued vector field can be assumed to be unit (see \cite{dan}), an anti-torqued hypersurface is also a constant angle hypersurface with anti-torqued axis.

When the ambient space is an arbitrary Riemannian manifold, constant angle surfaces  were also investigated in \cite{dgd} by assuming that $\v$ is parallel transported along any curve on the surface. In addition, submanifolds in arbitrary Riemannian spaces having constant ratio and principal direction properties with respect to an ambient vector field were also studied in \cite{mtv}. 

In view of the above discussion, we introduce the following notion.

\begin{definition}\label{d-pa}
Let $\Sigma$ be an orientable hypersurface in a Riemannian manifold $M$ and $\v$ be a nonzero vector field defined along $\Sigma$. Let $\theta:\Sigma\to [0,\frac{\pi}{2}]$ be a smooth function. We say that $\Sigma$ is a prescribed angle (PA) hypersurface in $M$ associated with the pair $(\v,\theta)$ if $\v$ is unit along $\Sigma$ and satisfies $\langle \v,\n\rangle=\cos \theta$, where $\n$ denotes the unit normal vector field of $\Sigma$. Such hypersurfaces will be briefly called PA hypersurfaces with $(\v,\theta)$. 
\end{definition}

Since the ambient space in Definition \ref{d-pa} is an arbitrary Riemannian manifold, unlike the usual approach in the literature, we no longer assume the existence of a globally defined vector field. Indeed, such an approach involves two separate problems: the existence of a globally defined vector field and the existence of PA hypersurfaces with respect to such a vector field. Instead, by Definition \ref{d-pa}, these two problems are reduced to a single problem. See Fig.~\ref{fig:comparison}.
\begin{figure}[ht!]
\centering
\resizebox{0.82\textwidth}{!}{%
\begin{tikzpicture}[
    >=Stealth,
    font=\small,
    title/.style={rectangle, rounded corners, draw=blue!60!black, thick, fill=blue!5,
        minimum width=5.6cm, minimum height=0.9cm, align=center, font=\bfseries},
    block/.style={rectangle, rounded corners, draw=blue!60!black, thick,
        minimum width=5.8cm, minimum height=1.5cm, text width=5.1cm, align=center},
    bigblock/.style={rectangle, rounded corners, draw=green!50!black, thick, fill=green!3,
        minimum width=6cm, minimum height=3.0cm, text width=5.2cm, align=center},
    arrow/.style={->, thick}
]
\node[title] (t1) at (0,2.5) {Existing approach};
\node[block] (b1) at (0,1)
{\textbf{1.} Existence of a globally\\ defined ambient vector field};
\node[block] (b2) at (0,-2)
{\textbf{2.} Existence of hypersurfaces\\ making a prescribed angle\\ with respect to this field};
\draw[arrow, shorten >=6pt, shorten <=6pt] (b1) -- (b2);
\node at (4.2,-0.5) {\Huge $\Longrightarrow$};
\node[title, draw=green!50!black, fill=green!5] (t2) at (8.5,2.5) {Our approach};
\node[bigblock] (b3) at (8.5,-0.7)
{Existence of a hypersurface \\[0.2cm] admitting a \textbf{unit} \\[0.2cm]\textbf{vector field}};
\end{tikzpicture}%
}
\caption{Comparison between the existing prescribed-angle approach and our approach.}
\label{fig:comparison}
\end{figure}
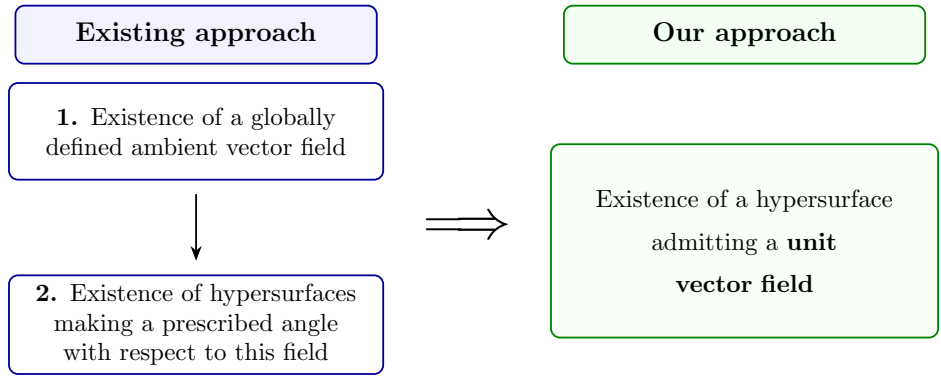

In this paper, we restrict our study to those PA hypersurfaces for which $\v$ is a torse-forming vector field. Therefore, whenever we refer to a PA hypersurface, we assume the existence of a unit torse-forming vector field along the hypersurface.


By a torse-forming vector field $\v$ along $\Sigma$, we mean a vector field $\v$ satisfying \eqref{int-tor}, where the covariant differentiation is taken with respect to the ambient Levi--Civita connection and $X\in \mathfrak{X}(\Sigma)$.
In this case, as we will see in Lemma \ref{l-par}, if $\v$ is a unit torse-forming vector field, then it must be anti-torqued along $\Sigma$. Definition \ref{d-pa} allows us to study PA hypersurfaces even when the ambient space does not admit a globally defined anti-torqued vector field (see Example \ref{ex31} and Remark \ref{rem31}). This is because a torse-forming vector field defined on an open subset cannot, in general, be extended to the whole ambient space \cite{amc,dta}. Therefore, assuming the existence of vector fields along hypersurfaces does not impose restrictions on the ambient space, in contrast to assuming the existence of globally defined ambient vector fields.


Moreover, as stated in Definition \ref{d-pa}, we do not impose the angle function $\theta$ to be constant; instead, we allow it to be an arbitrary function. This constitutes a distinguishing feature of our paper compared to the existing literature. Indeed, for example, on a totally geodesic PA hypersurface with $(\v,\theta)$, the condition $\theta \notin \{0,\pi/2\}$ implies that $\theta$ is constant if and only if $\v$ is parallel along the hypersurface (see Remark \ref{rem31}). Apart from this, no further information about the hypersurface can be concluded. However, when $\theta$ is allowed to vary, the intrinsic and extrinsic curvatures of a PA surface in a Riemannian manifold can be expressed in terms of the potential function of $\v$ and the angle function $\theta$ (see Proposition \ref{plevi}). Therefore,  it enables us to classify totally geodesic PA surfaces with constant intrinsic curvature (see Theorem \ref{classtotgeo}). Moreover, we also obtain classification and characterization results for non-totally geodesic PA surfaces when $\theta$ is constant (see Proposition \ref{zeroangle} and Theorems \ref{paral} and \ref{trightangle}).

Notice that our idea of allowing the angle function to vary is also inspired by the study of surfaces in $\r^3$ \cite{mun11}, where the tangential component of $\Phi/|\Phi|$ with respect to the surface is a principal direction and $\Phi$ denotes the position vector field. In this setting, the angle between $\Phi/|\Phi|$ and the unit normal vector varies. Recall that the vector field $\Phi/|\Phi|$ is an example of an anti-torqued vector field on $\r^n\setminus\{0\}$ \cite{dan}. Therefore, the considered surfaces in \cite{mun11} can be regarded as PA surfaces with $(\v,\theta)$ such that the tangential component of $\v$ is a principal direction.

The organization of the paper is as follows. In Section \ref{pre}, we give the preliminaries on hypersurfaces in Riemannian manifolds. In Section \ref{sec3}, we establish basic lemmas and results for PA hypersurfaces. Moreover, we study the particular cases of PA hypersurfaces with $(\v,\theta)$ when the function $\theta$ is constant or the vector field $\v$ is parallel along the hypersurface. In the final section, we obtain classification results for PA surfaces with constant intrinsic curvature. We also provide nontrivial examples corresponding to the obtained results.

\section{Preliminaries} \label{pre}

Let $(M^{n+1},\langle,\rangle)$ be a Riemannian manifold of dimension $n+1$. When the dimension is not necessary to specify, we will denote the manifold simply by $M$.

Let $\pi\subset T_pM$ be  a plane section at a point $p \in M$ and $\{e_1,e_2\}$ be a basis of $\pi$. The sectional curvature of $\pi$ is defined by
$$
K(\pi)=\frac{\langle R(e_1,e_2)e_2,e_1\rangle}{\langle e_1,e_1\rangle \langle e_2,e_2\rangle-\langle e_1,e_2\rangle^2},
$$
where $R$ is the Riemannian curvature tensor of $M$ associated with the Levi–Civita connection. The manifold $M$ is called a space form if $K$ is a real constant for all plane sections and points. 

Let $\Sigma$ be an orientable hypersurface in $M^{n+1}$ and $\n$ its unit normal vector field. The shape operator $A$ of $\Sigma$ is a symmetric $(1,1)$-tensor field  on $T\Sigma$ defined by $A(X)=-\nabla_X \n$, $X\in \x(\Sigma)$, where $\nabla$ is the Levi–Civita connection on $M^{n+1}$. If $M^{n+1}$ and $\Sigma$ are orientable and equipped with chosen orientations, then, since $A$ is symmetric, there exists, at each point $p\in\Sigma$, an orthonormal basis of principal directions $\{e_1,\dots,e_n\}\subset T_p\Sigma$ with corresponding real principal curvatures $\kappa_1,\dots,\kappa_n$ \cite{car}. 

The induced Levi-Civita connection $\nabla^{\Sigma}$ on $\Sigma$ is obtained from the Gauss formula:
$$
\nabla_XY=\nabla^{\Sigma}_XY+\langle A(X),Y\rangle \n, \quad X,Y \in\x(\Sigma).
$$
The hypersurface $\Sigma$ is said to be totally geodesic if $A=0$ and totally umbilical if $A=\delta \mbox{Id}$, for a smooth function $\delta$ on $\Sigma$. For every $X,Y,Z\in \mathfrak{X}(\Sigma)$, the Gauss equation is given by
$$
R^\Sigma(X,Y)Z
=
\big(R(X,Y)Z\big)^\top
+
\langle A(Y),Z\rangle A(X)
-
\langle A(X),Z\rangle A(Y),
$$
where $(\cdot)^\top$ denotes the tangential component along $\Sigma$.

In particular, if $\{e_1,e_2\}$ is an orthonormal frame on a surface $\Sigma\subset M^3$, then the Gauss equation yields
$$
K^\Sigma(e_1,e_2)=K(e_1,e_2)+\langle A(e_1),e_1\rangle \langle A(e_2),e_2\rangle-\langle A(e_1),e_2\rangle^2,
$$
where $K^\Sigma(e_1,e_2)$ denotes the intrinsic curvature of the surface and $K(e_1,e_2)$ is the sectional curvature of the tangent plane $T_p\Sigma \subset T_pM^3$. Consequently, the Gauss equation reduces to $
K_{\mathrm{int}}=K_{\mathrm{sec}}+K_{\mathrm{ext}}, $ where $K_{\mathrm{ext}} = \det A$ is the extrinsic curvature of $\Sigma$. We call $\Sigma$  is extrinsically (resp. intrinsically) flat if $K_{\mathrm{ext}}$ (resp. $K_{\mathrm{int}}$) vanishes identically.

A surface $\Sigma \subset M^3$ is said to be ruled if it admits a foliation by curves which are geodesics of the ambient space \cite{lrs}.

\section{PA hypersurfaces with torse-forming vector fields} \label{sec3}

In this section, we study PA hypersurfaces Riemannian manifolds associated with torse-forming vector fields. We first characterize totally-umbilical hypersurfaces using concircular vector fields.

\begin{proposition}\label{p-umb}
There exists a unit concircular vector field $\v$ along an orientable hypersurface $\Sigma$ with nonzero potential function if and only if $\Sigma$ is totally umbilical with normal vector field $\v$.
\end{proposition}
\begin{proof}
Denote by $\nabla$ the Levi-Civita connection of the ambient space. Assume that $\v$ is a unit concircular vector field along $\Sigma$, that is, $\nabla_X \v = f X$ for every $X\in \x(\Sigma)$. Since $\v$ has unit length, we have
$$
0=X\langle \v, \v \rangle = 2\langle \nabla_X \v, \v\rangle = 2f \langle \v, X \rangle,
$$
for every $X\in \x(\Sigma)$. Since $f \neq 0$, it follows that $\langle \v, X \rangle = 0$ and hence $\v$ is normal to $\Sigma$. Therefore, $\v = \pm \n$, where $\n$ is the unit normal vector field. Consequently, $\nabla_X \n =\pm f X,$ which implies $A = \mp f \operatorname{Id}$, namely, $\Sigma$ is totally umbilical.

Conversely, let $\Sigma$ be a totally umbilical hypersurface and $\n$ its unit normal vector field. Since $\Sigma$ is totally umbilical, we have $\nabla_X\n=-\lambda X$,  $X\in \x(\Sigma)$, for some nonzero smooth function $\lambda$. This yields that $\n$ is a unit concircular vector field along $\Sigma$ with potential function $f=-\lambda$. 
\end{proof}

\begin{remark}
Proposition \ref{p-umb} generalizes \cite[Proposition 4]{loy5}, where the result was established for real space forms, to arbitrary Riemannian manifolds.
\end{remark}

Let $\Sigma$ be a  hypersurface in a Riemannian manifold $M$ and $\v$ be a unit torse-forming vector field along $\Sigma$. Denote by $\nabla$ the Levi-Civita connnection on $M$. Then, for any $X\in \x(\Sigma)$ we write
$$
0=X\langle \v, \v \rangle = 2\langle \nabla_X \v, \v\rangle = 2 \langle fX+\omega(X)\v, \v \rangle,
$$
or equivalently 
\begin{equation}\label{omega}
\omega(X)=-f\langle X, \v \rangle, \quad X\in \x(\Sigma).
\end{equation}
This yields the following result:

\begin{lemma}\label{l-par}
Let $\v$ be a unit torse-forming vector field along a hypersurface $\Sigma$. Then, the vector field $\v$ is anti-torqued along $\Sigma$. Moreover, the potential function of $\v$ vanishes along $\Sigma$ if and only if $\v$ is parallel along $\Sigma$.
\end{lemma}

By Lemma \ref{l-par}, in our setting, natural candidates for PA hypersurfaces are those in ambient spaces admitting globally defined anti-torqued vector fields. Indeed, if $M$ is such a Riemannian manifold, then every orientable hypersurface in $M$ is a PA hypersurface.

More generally, once the ambient space admits a totally umbilical hypersurface, one obtains an immediate example of a PA hypersurface, as shown below.

\begin{example}\label{ex31}
Let $M$ be a Riemannian manifold admitting a totally umbilical hypersurface $\Sigma$ and $\n$ its unit vector field. Then, $\Sigma$ is an example of a PA hypersurface with $(\v,\theta)$, where $\v=\pm\n$ and $\theta=0$. Indeed, since $\Sigma$ is totally umbilical, we have $\nabla_X \n = \lambda X$ for some smooth function $\lambda$ on $\Sigma$ and for all $X\in \mathfrak{X}(\Sigma)$. Hence,
$$\nabla_X\v=\lambda X=\lambda(X-\langle X,\v\rangle\v), \quad X\in \x(\Sigma),$$
which means that the normal vector field $\v=\pm\n$ acts as a unit anti-torqued vector field along $\Sigma$.
\end{example}

\begin{remark}\label{rem31}
Example \ref{ex31} shows that a PA hypersurface can be considered in a Riemannian manifold even if it does not admit a globally defined anti-torqued vector field; for instance, when the ambient space is a sphere (see \cite[Example 1]{dan}).
\end{remark}

Assume that $\Sigma$ is a PA hypersurface with $(\v,\theta)$ in a Riemannian manifold $M$. Denote by $\n$ the unit normal vector field to $\Sigma$. By Definition \ref{d-pa}, the vector field $\v$ admits a  decomposition into its tangential and normal components along $\Sigma$, given by
\begin{equation}\label{decomp}
\v=\t +\cos \theta \n, \quad  \t=\sin\theta e_1,
\end{equation}
where $e_1$ is a unit tangent vector field to $\Sigma$ and $\theta:\Sigma\to [0,\frac{\pi}{2}]$ is a smooth function. 

Let $\nabla$ and $\nabla^\Sigma$ denote the Levi-Civita connections of $M$ and $\Sigma$, respectively. By the Gauss formula, it follows that
\begin{equation}\label{der1}
\nabla_X\v=\nabla^\Sigma_X\t-\cos \theta A(X)+ \big( \langle A(\t),X\rangle+ X(\cos \theta) \big)\n,
\end{equation}
for every $X\in \x(\Sigma)$. On the other hand, by Lemma \ref{l-par}, $\v$ is anti-torqued along $\Sigma$, i.e., $\nabla_X\v=f\big(X-\langle X,\v\rangle\v\big).$ Using the decomposition \eqref{decomp} gives
$$
\nabla_X\v=f\big(X-\langle X,\t\rangle(\t+\cos \theta\n)  \big), \quad X\in\x(\Sigma).
$$
Comparing with \eqref{der1}, we derive
\begin{equation}\label{der-comp}
\left\{
\begin{aligned}
f\big(X-\langle X,\t\rangle \t\big)
&=\nabla^\Sigma_X\t-\cos \theta A(X),\\[4pt]
\langle A(\t),X\rangle+ X(\cos \theta)
&=-f \cos \theta \langle X,\t\rangle.
\end{aligned}
\right.
\end{equation}

Therefore, by the second equation of \eqref{der-comp}, we have the  following result

\begin{proposition}\label{con-prin}
Let $\Sigma$ be a PA hypersurface with $(\v,\theta)$, where $\v$ is an anti-torqued vector field with potential function $f$ and $\theta\neq 0$. Let also $\t$ be the tangential component of $\v$ with respect to $\Sigma$. Then,
\begin{enumerate}
\item[(i)] $\t$ is a principal direction if and only if the function $\theta$ is constant along the orthogonal distribution to $\t$.
\item[(ii)] $\t$ is a principal direction with principal curvature $\kappa_1=-f\cos\theta$ if and only if $\theta$ is constant.
\end{enumerate}
\end{proposition}
\begin{proof}
Since $\theta\neq 0$, we have $\t\neq 0$. Let $Y\in\x(\Sigma)$ be orthogonal to $\t$. Then
the second equation of \eqref{der-comp} gives $\langle A(\t),Y\rangle-\sin\theta Y(\theta)=0.$ Therefore $\t$ is a principal direction if and only if $\theta$ is constant along the orthogonal disribution to $\t$, completing the proof of the first item. For the second item, assume that $\t$ is a principal direction of $\Sigma $, say $A(\t)=\kappa_1 \t$. Then, the second equation of \eqref{der-comp} yields
\begin{equation}\label{con-prin0}
X(\cos\theta)=-(\kappa_1+f\cos\theta)\langle X,\t\rangle, \quad X\in \x(\Sigma).
\end{equation}
It follows that $\theta$ is constant on $\Sigma$ only if $\kappa_1=-f\cos\theta$. 
\end{proof}

This result extends \cite[Proposition 2.1]{mun11}, obtained in the Euclidean space, to arbitrary Riemannian manifolds, where the position vector field of the Euclidean space is replaced by an arbitrary anti-torqued vector field. Moreover, in the following remark, we indicate the differences between working with torse-forming vector fields and parallel vector fields, as in \cite{dgd}.

\begin{remark}\label{rem31}
By Proposition \ref{con-prin}, the condition that $\t$ is a principal direction is not sufficient by itself to ensure that $\theta$ is constant. In addition, the corresponding principal curvature must satisfy $\kappa_1=-f\cos\theta$. This indicates the difference between working with torse-forming vector fields and parallel vector fields (see \cite[Proposition 5]{dgd}).  Another difference can be seen as follows. To clarify this, consider the simplest case where $\Sigma$ is totally geodesic. Then $A=0$ and the second equation in \eqref{der-comp} reduces to
$$
\sin\theta\, X(\theta)= f \cos\theta \sin\theta \langle X,e_1\rangle.
$$
This implies that $\theta$ must be constant as long as $f$ vanishes. Otherwise, i.e., if $f\neq 0$, the function $\theta$ is not necessarily constant.
\end{remark}

In the last results of this section, we study the particular cases where $\v$ is parallel and where either the tangential or the normal component of $\v$ vanishes identically. The first result extends \cite[Proposition 5]{dgd} to arbitrary dimensions. For the second, see also \cite[Proposition 4.1, Corollary 4.2]{amc1}.

\begin{theorem}\label{paral}
Let $\Sigma$ be a non-totally geodesic PA hypersurface with $(\v,\theta)$, where $\v$ is a parallel vector field. 
If $\theta$ is constant, then the integral curves of the tangential component of $\v$ are ambient geodesics. Moreover,  the shape operator of $\Sigma$ satisfies $\operatorname{det}(A)=0$.
\end{theorem}
\begin{proof}
Since $\v$ is a parallel vector field along $\Sigma$, by \eqref{der-comp} it follows
$$
\nabla^\Sigma_X\t=\cos \theta A(X), \quad \langle A(\t),X\rangle=-X(\cos \theta), \quad X\in \x(\Sigma).
$$
If $\theta$ is constant, then we get $A(\t)=0$ and $\nabla^\Sigma_{\t}\t=0$. Hence by the Gauss fomrula, we derive $\nabla_{\t}\t=0$, completing the result of the first part. To complete the proof, let $\{e_1,e_2,\dots,e_n\}$ be a local orthonormal frame of principal directions on $\Sigma$. In terms of this frame, the matrix of the shape operator takes the form $A=0\oplus\operatorname{diag}[\kappa_2,\ldots,\kappa_n]$, where $\oplus$ denotes the direct sum of matrices. Hence its determinant is zero.
\end{proof}

\begin{proposition}\label{zeroangle}
Let $\Sigma$ be a PA hypersurface with $(\v,\theta)$, where $\v$ is an anti-torqued vector field.  
\begin{enumerate}
\item[(i)] If $\theta=\frac{\pi}{2}$, then the integral curves of the tangential component of $\v$ are ambient geodesics. Moreover,  the shape operator of $\Sigma$  satisfies $\operatorname{det}(A)=0$.
\item[(ii)] If $\theta=0$, then $\Sigma$ is totally umbilical. 
\end{enumerate}
\end{proposition}
\begin{proof} 
\begin{enumerate}
\item[(i)] Writing $\theta=\frac{\pi}{2}$ in \eqref{decomp} yields $\t=e_1$. Considering in \eqref{der-comp}, we find 
$$
\nabla^\Sigma_X\t=f\big(X-\langle X,\t\rangle \t\big) , \quad \langle A(\t),X\rangle=0, \quad X\in \x(\Sigma).
$$
This gives $A(\t)=0$ and $\nabla^\Sigma_{\t}\t=0$. The proof is completed by following that of Theorem \ref{paral}.
\item[(ii)]  If $\theta=0$, then by \eqref{decomp} it follows that $\t=0$. The conclusion follows from the first equality in \eqref{der-comp}.
\end{enumerate}
\end{proof}

Considering item (ii) of Proposition \ref{zeroangle} together with Example \ref{ex31} and comparing them with Proposition \ref{p-umb}, an important difference emerges. The existence of a unit concircular vector field $\v$ along a hypersurface forces the hypersurface to be totally umbilical and $\v$ becomes normal to the hypersurface. However, the existence of a unit anti-torqued vector field on a hypersurface does not, in general, imply that the hypersurface is totally umbilical. For instance, as will observe in the items of Example \ref{ex41}, the surfaces admitting a unit anti-torqued vector field are not totally umbilical.

\section{Constant curvature PA surfaces} \label{sec4}

In this section, we study the curvature properties of PA surfaces in $3$-dimensional Riemannian manifolds. 

Under the assumption that the tangential component of $\v$ is a principal direction (or, in the sense of \cite{mtv}, that the PA surface has the principal direction property with respect to $\v$), we first establish an auxiliary result for computing the Levi-Civita connection and the intrinsic curvature of such surfaces.

\begin{proposition}\label{plevi}
Let $\Sigma$ be a PA surface with $(\v,\theta)$, where $\v$ is an anti-torqued vector field with potential function $f$ and $\theta\neq0$. Let also $\{e_1,e_2\}$ be a local orthonormal frame on $\Sigma$ of principal directions and $\kappa_i$ the corresponding principal curvatures. 
If the tangential component of $\v$ satisfies $\t=\sin \theta e_1$, then 
\begin{enumerate}
\item[(i)]
\begin{equation}\label{e12}
e_1(\theta)=\kappa_1+f\cos\theta, \quad e_2(\theta)=0,
\end{equation}
\item[(ii)] 
\begin{equation}\label{levi}
\nabla^\Sigma_{e_1}e_1=\nabla^\Sigma_{e_1}e_2=0, \quad \nabla^\Sigma_{e_2}e_1=Be_2, \quad \nabla^\Sigma_{e_2}e_2=-Be_1, \quad [e_1,e_2]=-Be_2,
\end{equation}
\item[(iii)] there is a coordinate system $(u,v)$ on $\Sigma$ such that its metric is given by
\begin{equation}\label{intcur0}
g^\Sigma=du^2+\big(pe^{\int Bdu)}\big)^2dv^2, \quad e_1=\partial_u, \quad e_2=\big(pe^{\int Bdu)}\big)^{-1}\partial_v,
\end{equation}
\item[(iv)] the curvatures of $\Sigma$ with respect to $(u,v)$ are
\begin{equation}\label{intcur}
K_{\mathrm{int}}=-\left(\frac{\partial B}{\partial u}+B^2\right), \quad K_{\mathrm{ext}}=\left(\frac{d\theta}{du}-f\cos\theta\right)\kappa_2,
\end{equation}
where $B:=\frac{f+\kappa_2\cos\theta}{\sin \theta}$ and $p=p(v)$ is some positive function. 
\end{enumerate}
\end{proposition}
\begin{proof}
\begin{enumerate}
\item[(i)] Writing $X=e_1$ and $X=e_2$ into \eqref{con-prin0} yields the derivatives $e_i(\theta)$ in \eqref{e12}.
\item[(ii)] We compute the covariant derivatives $\nabla^\Sigma_{e_i}e_j$. By plugging $\t=\sin\theta e_1$ and $A(e_i)=\kappa_ie_i$ into \eqref{der-comp} we obtain
\begin{equation}\label{der-comp3}
\left\{
\begin{aligned}
f\big(X-\sin^2 \theta\langle X,e_1\rangle e_1\big)
&=X(\sin \theta)e_1+\sin \theta\nabla^\Sigma_Xe_1-\cos \theta A(X),\\[4pt]
\kappa_1\sin \theta\langle e_1,X\rangle+ X(\cos \theta)
&=-f \sin \theta\cos \theta \langle X,e_1\rangle.
\end{aligned}
\right.
\end{equation}
Substituting $X=e_1$ in the first equation of \eqref{der-comp3} and considering \eqref{e12}, we conclude that $\sin \theta\nabla^\Sigma_{e_1}e_1=0$. From the identity $0=e_1\langle e_1,e_2\rangle$, it follows that $\langle e_1,\nabla^\Sigma_{e_1}e_2\rangle=0$. Since $\nabla^\Sigma_{e_1}e_2$ is orthogonal to both $e_1$ and $e_2$, we derive $\nabla^\Sigma_{e_1}e_2=0$. In addition, writing $X=e_2$ in the first equation of \eqref{der-comp3}, we obtain $ \nabla^\Sigma_{e_2}e_1=Be_2$. The value of the covariant derivative $\nabla^\Sigma_{e_2}e_2$  follows from the identity $\langle \nabla^\Sigma_{e_2}e_1,e_2\rangle=- \langle \nabla^\Sigma_{e_2}e_2,e_1\rangle.$ Since $\nabla^\Sigma$ is the Levi-Civita connection,  the Lie bracket $[e_1,e_2]$ is concluded by the torsion-free property. Therefore, we have obtained the expressions in \eqref{levi}.
\item[(iii)] Using the items (i) and (ii), the curvature tensor is computed by 
\begin{eqnarray*}
R^\Sigma(e_2,e_1)e_1&=&\nabla^\Sigma_{e_2}\nabla^\Sigma_{e_1}e_1-\nabla^\Sigma_{e_1}\nabla^\Sigma_{e_2}e_1+\nabla^\Sigma_{[e_1,e_2]}e_1 \\
&=& -\big(e_1(B)+B^2\big)e_2,
\end{eqnarray*}
which yields $K_{\mathrm{int}}=-\big(e_1(B)+B^2\big)$. Again by using $[e_1,e_2]=-Be_2$ in \eqref{levi}, we may choose a local coordinate system $(u,v)$ on $\Sigma$ such that $e_1=\partial_u$ and $e_2=\frac{1}{\beta(u,v)}\partial_v$, where $\beta(u,v)$ is a smooth positive function. Since $[e_2,e_1]=B e_2$, it follows that $B=\frac{\beta_u}{\beta}$. Integrating, we obtain $\beta(u,v)=p(v)e^{\int B(u,v)du}$ for some positive function $p(v)$. 
\item[(iv)] $K_{\mathrm{int}}$ is already been computed in the previous item. The computation of $K_{\mathrm{ext}}$ follows from the first equality in \eqref{e12} together with the product of the principal curvatures.
\end{enumerate}
\end{proof}

In the next two subsections, we separately investigate the cases where $\Sigma$ is totally geodesic or not.

\subsection{Totally geodesic PA surfaces} 
Assume that $\Sigma$ is totally geodesic and $\theta \neq 0,\frac{\pi}{2}$. As emphasized in Remark \ref{rem31}, if $\v$ is not parallel then the prescribed angle $\theta$ is not necessarily constant. In this case, Proposition \ref{plevi} yields the following result.

\begin{proposition}\label{tg}
Let $\Sigma$ be a totally geodesic PA surface with $(\v,\theta)$, where $\v$ is an anti-torqued vector field with potential function $f$ and $\theta$ is non-constant. Then there is a local coordinate system $(u,v)$ on $\Sigma$ such that
\begin{equation}\label{tg1}
\theta(u)=\cos^{-1}\big( \sech F(u) \big), \quad K_{\mathrm{int}}=-\big(F''(u)\coth F(u)+F'(u)^2\big).
\end{equation}
where $F(u):=\int^uf(t)dt+c$, $c\in\r$.
\end{proposition}
\begin{proof}
Since $\Sigma$ is totally geodesic, its principal curvatures vanish and hence every tangent direction is principal. Then, by \eqref{e12} and \eqref{intcur0}, we have $\theta_u=f\cos \theta$. Integrating this equation yields the first identity in \eqref{tg1}. On the other hand, the function $B$ introduced in Proposition \ref{plevi} is now $B=f\csc\theta$. Due to $\csc\theta=\coth F$, it follows that $B=f\coth F$. Substituting this expression into \eqref{intcur}, we obtain the formula for $K_{\mathrm{int}}$, completing the proof.
\end{proof}

Notice that the potential function $f$ and the intrinsic curvature $K_{\mathrm{int}}$ can be also expressed in terms of  $\theta$ by
\begin{equation*}
f=\theta' \sec\theta, \quad K_{\mathrm{int}}=-\big((\log \tan \theta)''+(\log \tan\theta)'^2\big).
\end{equation*}

We now provide an example in the hyperbolic ambient space illustrating Proposition \ref{tg}.

\begin{example}\label{htg}
Consider the upper half-space model of $\h^3$
$$
\h^3=\{(x,y,z)\in \r^3:z>0\}
$$
endowed with the conformal metric $\langle,\rangle_h=\frac{1}{z^2}\langle,\rangle_e$, where $\langle,\rangle_e$ is the Euclidean metric. It is known from \cite{amc} that $\v=-z\partial_z$ is an anti-torqued vector field on $\h^3$ whose potential function satisfies $f=1$.

Let $\s^2_+$ denote  the hemi-sphere contained in $\h^3$ given by
$$
\s^2_+=\{(x,y,z)\in \r^3:x^2+y^2+z^2=r^2,z>0\}.
$$
The principal curvatures $\kappa_i^h$ of $\s^2_+$ in $\h^3$ are related to the Euclidean principal curvatures $\kappa_i^e$ by $\kappa_i^h=z\kappa_i^e+\langle \n_e,\partial_z\rangle$, where $\n_e$ denotes the Euclidean unit normal to $\s^2_+$ in $\r^3$ (see \cite{gss}). Since $\kappa_i^e=-\frac1r$ and $\langle \n_e,\partial_z\rangle=\frac{z}{r}$, we have $\kappa_i^h=0$. Hence, $\s^2_+$ is totally geodesic in $\h^3$. 

The angle $\theta$ between the unit normal $\n_h$ of $\s^2_+$ and the anti-torqued vector field $\v$ along $\s^2_+$ is given by $\cos\theta=\frac{z}{r}$. In addition, since $\h^3$ has constant sectional curvature $-1$, the intrinsic curvature of $\s^2_+$ is $K_{\mathrm{int}}=-1$.

Now choose a local coordinate system $(u,v)$ on $\s^2_+$ given by
$$
x=r\cos g(u)\cos v, \quad y=r\cos g(u)\sin v, \quad z=r\sin g(u),
$$
where $g(u)=\big(\sin^{-1}(\sech u)\big)$. Due to $f=1$ and, up to an integration constant, the function $F$ given in \eqref{tg1} is $F(u)=u$. Consequently, the angle function satisfies $\cos\theta(u)=\sech u$, verifying the first expression in \eqref{tg1}. Moreover, since $F''=0$, $K_{\mathrm{int}}=-1$ which also verifies the intrinsic curvature formula in \eqref{tg1}.
\end{example}

In Example \ref{htg}, since $\h^3$ has consant sectional curvature, the intrinsic curvature $K_{\mathrm{int}}$ of the totally geodesic surface $\s^2_+$ becomes directly $-1$. However, this is not the case in general. Notice that the second equality in \eqref{tg1} is indeed an ODE and the existence of solutions is guaranteed by the standard theory of ODEs.

In the following result, we provide a necessary and sufficient condition for a totally geodesic PA surface to have constant intrinsic curvature, expressed in terms of the potential function.

\begin{theorem}\label{classtotgeo}
Let $\Sigma$ be a totally geodesic PA surface with $(\v,\theta)$, where $\v$ is an anti-torqued vector field with nonzero potential function $f$ and $\theta$ is non-constant. The intrinsic curvature $K_{\mathrm{int}}$ is some constant $K_0$ if and only if there is a coordinate system $(u,v)$ on $\Sigma$ such that 
\begin{equation}\label{ODEtg0}
F(u)=
\begin{cases}
\displaystyle
\sinh^{-1}\!\left(
\sqrt{\frac{c_1-K_0}{-K_0}}\,
\sinh\big(\pm\sqrt{-K_0}(u+c_2)\big)
\right) , & K_0<0, \, c_1\geq 0,\\[14pt]
\displaystyle
\sinh^{-1}\big(\pm\sqrt{c_1}\,u+c_2\big), & K_0=0, \, c_1>0,\\[10pt]
\displaystyle
\sinh^{-1}\!\left(
\sqrt{\frac{c_1-K_0}{K_0}}\,
\sin\big(\sqrt{K_0}(u+c_2)\big)
\right), & 0<K_0<c_1, 
\end{cases}
\end{equation}
where $c_1,c_2\in \r$  and the potential function is given by $f=F'$.
\end{theorem}
\begin{proof}
Assume that $K_{\mathrm{int}}=K_0$. Then, by \eqref{tg1}, we obtain
\begin{equation}\label{ODEtg}
F''\coth F+F'^2+K_0=0.
\end{equation}
Since $f\neq 0$, the function $F$ is non-constant. If $F''=0$, then \eqref{ODEtg} reduces to $F'^2+K_0=0$, which implies that $K_0<0$ and $F'(u)=\pm\sqrt{-K_0}$. This corresponds to a particular solution of the form $F(u)=\pm\sqrt{-K_0}u+c$ in \eqref{ODEtg0}, where $c_1=0$ and $c_2=c$.

We now assume that $F''\neq 0$ and proceed to solve \eqref{ODEtg}. Since $F''\neq 0$, we may treat the function $F'$ as a variable. Therefore, we set $F'=P$. By the chain rule, differentiating this equality yields
$$F''=\frac{dF}{du}\frac{dP}{dF}=P\frac{dP}{dF}.$$ 
Substituting into \eqref{ODEtg}, we conclude
$$
P\frac{dP}{dF}\coth F+P^2+K_0=0.
$$
We change the variable again and set $Q=P^2$. Differentiating, we obtain $\frac{dQ}{dF}=2P\frac{dP}{dF}$, and thus
$$
\frac{1}{Q+K_0}\frac{dQ}{dF}=-2\tanh F.
$$
Integrating,

\begin{equation}\label{ODEtg001}
P^2=c_1\sech^2F-K_0,
\end{equation}
where $c_1>0$ is some constant.

Finally, we use another change of variable. Introduce $S(u)=\sinh F(u)$. Then $S'=F'\cosh F$. Multiplying \eqref{ODEtg001} by $\cosh^2 F=1+\sinh^2 F$, we obtain
\begin{equation}\label{ODEtg1}
S'^2=c_1-K_0(1+S^2)=c_1-K_0-K_0S^2.
\end{equation}
The solutions depend on the sign of $K_0$.

\textbf{Case $K_0<0$.} Solving \eqref{ODEtg1}, we obtain
$$
S(u)=\sqrt{\frac{K_0-c_1}{K_0}}\sinh \big (\pm\sqrt{-K_0}u+c_2 \big ), \quad c_2\in \r.
$$
Since $S(u)=\sinh F(u)$, this yields the expression for $F(u)$ in the case $K_0<0$.

\textbf{Case $K_0=0$.} Equation \eqref{ODEtg1} reduces to $S'^2=c_1$ and so $S(u)=\pm \sqrt{c_1}u+c_2$, $c_2\in \r.$ Thus $F(u)=\sinh^{-1}(\sqrt{c_1}u+c_2)$. 

\textbf{Case $K_0>0$.} In this case, \eqref{ODEtg1} implies $c_1>K_0$ and we derive
$$
S(u)=\sqrt{\frac{c_1-K_0}{K_0}}\sin \big (\sqrt{K_0}u+c_2 \big ), \quad c_2\in \r.
$$
Again, using $S(u)=\sinh F(u)$, we obtain the required expression for $F(u)$ in the case $K_0>0$.
\end{proof}

Once Theorem \ref{classtotgeo} has been proved, we may conclude from \eqref{ODEtg0} that the potential function $f$ is given by
\begin{equation}\label{potf0}
f(u)=F'(u)=
\begin{cases}
\displaystyle
\frac{\pm\sqrt{c_1-K_0}\,
\cosh\big(\pm\sqrt{-K_0}(u+c_2)\big)}
{\left(1+\frac{c_1-K_0}{-K_0}
\sinh^2\big(\pm\sqrt{-K_0}(u+c_2)\big)\right)^{1/2}} ,
& K_0<0, \, c_1\geq 0,\\[16pt]
\displaystyle
\frac{\pm\sqrt{c_1}}
{\left(1+\big(\pm\sqrt{c_1}\,u+c_2\big)^2\right)^{1/2}},
& K_0=0, \, c_1>0,\\[16pt]
\displaystyle
\frac{\sqrt{c_1-K_0}\,
\cos\big(\sqrt{K_0}(u+c_2)\big)}
{\left(1+\frac{c_1-K_0}{K_0}
\sin^2\big(\sqrt{K_0}(u+c_2)\big)\right)^{1/2}} ,
& 0<K_0<c_1.
\end{cases}
\end{equation}

Returning to Example \ref{htg}, the case $K_{\mathrm{int}}=-1$ corresponds to the choice $c_1=0$ in the family of solutions with $K_{\mathrm{int}}<0$ in \eqref{potf0}. 

Next, consider the horizontal plane $\Sigma$ in $\r^3$ given by $\Sigma:\{(x,y,z)\in \r^3:z=1\},$ which is a totally geodesic surface. A polar coordinate system on $\Sigma$ is given by $(x,y)=u(\cos v,\sin v)$. Let $\Phi$ be the position vector. Then, $\v=\frac{\Phi}{|\Phi|}$ is an anti-torqued vector field along $\Sigma$ with potential function $f=\frac{1}{|\Phi|}=\frac{1}{\sqrt{1+u^2}}$, which corresponds to the case $K_{\mathrm{int}}=0$ in \eqref{potf0}, where $c_1=1$ and $c_2=0$.

\subsection{Non-totally geodesic PA surfaces}
Even if $\Sigma$ is not totally geodesic, it is still possible to characterize such surfaces under the assumption of constant intrinsic curvature. Let $\Sigma$ be a PA surface satisfying the conditions of Proposition \ref{plevi}. If the intrinsic curvature $K_{\mathrm{int}}$ is a constant $K_0$, then from \eqref{intcur} we obtain
$$
B_u+B^2+K_0=0.
$$
According to the sign of $K_0$, its solutions are given by
\begin{equation}\label{solvb1}
B(u,v)=\frac{f+\kappa_2\cos\theta}{\sin\theta}=
\begin{cases}
\sqrt{K_0}\,\tan\big(q(v)-\sqrt{K_0}\,u\big), & K_0>0,\\[6pt]
\dfrac{1}{u+q(v)}, & K_0=0,\\[10pt]
\sqrt{-K_0}\,\tanh\big(q(v)+\sqrt{-K_0}\,u\big) \quad \textnormal{or} \quad
\pm\sqrt{-K_0},  & K_0<0,
\end{cases}
\end{equation}
where $q(v)$ is an arbitrary smooth function.

Particularly, it is worth investigating the case $\theta=\frac{\pi}{2}$, because in this case the potential function $f$ coincides with the function $B$. This leads to the following result.

\begin{theorem}\label{trightangle}
Let $\Sigma$ be a PA surface with $(\v,\theta)$, where $\v$ is an anti-torqued vector field with nonzero potential function $f$.  If $\theta=\frac{\pi}{2}$, then $\Sigma$ is an extrinsically flat ruled surface. Moreover, the intrinsic curvature $K_{\mathrm{int}}$ is a constant  $K_0$ if and only if there is a coordinate system $(u,v)$ on $\Sigma$ such that
\begin{equation}\label{solvf1}
f(u,v)=
\begin{cases}
\sqrt{K_0}\,\tan\big(q(v)-\sqrt{K_0}\,u\big), & K_0>0,\\[6pt]
\dfrac{1}{u+q(v)}, & K_0=0,\\[10pt]
\sqrt{-K_0}\,\tanh\big(q(v)-\sqrt{-K_0}\,u\big) \quad \textnormal{or} \quad
\pm\sqrt{-K_0},  & K_0<0,
\end{cases}
\end{equation}
where $q(v)$ is an arbitrary smooth function.
\end{theorem}

\begin{proof}
Due to $\theta=\frac{\pi}{2}$, it follows from \eqref{decomp} that $\v$ is tangent to $\Sigma$ and defines a principal direction with vanishing principal curvature. This yields that $K_{\mathrm{ext}}=0$ and that the hypothesis of Proposition \ref{plevi} is satisfied.  Moreover, by Proposition \ref{plevi}, we have $B=f$. Therefore, the result follows from \eqref{solvb1}.
\end{proof}

Assume that $\Sigma$ is a PA surface with $(\v,\frac{\pi}{2})$. From Theorem \ref{trightangle}, it follows that $K_{\mathrm{ext}}=0$. By the Gauss equation, if the ambient space is a space form, then the sectional curvature $K_{\mathrm{sec}}$ is constant and consequently the intrinsic curvature $K_{\mathrm{int}}$ is also constant. For instance, we have the following:

\begin{example}\label{ex41}
\begin{enumerate}
\item[(i)] Let $\Phi$ be the position vector field on $\mathbb{R}^n$. Then $\v=\frac{\Phi}{|\Phi|}$ defines an anti-torqued vector field on $\r^n\setminus \{0\}$. Consider a cone $\Sigma \subset \r^3$ with vertex at the origin. The vector field $\v$ is tangent to $\Sigma$. In our setting, $\Sigma$ is a PA surface with $(\v,\frac{\pi}{2})$. A local parametrization of $\Sigma$ is given by $X(u,v)=u\beta(v)$, where $\beta\subset \s^2$ is an arclength parametrized curve. Since the potential function of $\v$ is $f=\frac{1}{|\Phi|}$, its restriction to $\Sigma$ satisfies $(f\circ X)(u,v)=\frac1u$. This is consistent with \eqref{solvf1} in the case $K_0=0$.
\item[(ii)] Similar to Example \ref{htg}, we consider the upper half-space model of $\mathbb{H}^3$. Let $\v=-z\partial_z$ be an anti-torqued vector field on $\mathbb{H}^3$ with potential function equal to $1$. Then a cylinder $\Sigma$, whose rulings are parallel to $\partial_z$, is a PA surface with $(\v,\frac{\pi}{2})$. This corresponds to the case $K_0=-1$ in \eqref{solvf1}.
\end{enumerate}
\end{example}

As a final observation, we consider the case where the vector field $\v$ is parallel along $\Sigma$. If $\Sigma$ is totally umbilical, then every tangent direction to $\Sigma$ is principal and the principal curvatures coincide, say $\kappa\neq 0$. Hence, the assumption of Proposition \ref{plevi} is satisfied. It follows from \eqref{e12} that $\kappa=\theta'$ and $B=\theta'\cot\theta.$ Substituting into \eqref{intcur}, we get
$$
K_{\mathrm{int}}=\theta''\cot\theta+\theta'^2, \quad K_{\mathrm{ext}}=\theta'^2,
$$
and therefore $K_{\mathrm{sec}}=\theta''\cot\theta.$ In particular, if $K_{\mathrm{ext}}$ is constant, then $\theta'$ is constant and hence $\theta''=0$. Consequently,
$$
K_{\mathrm{sec}}=0 \quad \text{and} \quad K_{\mathrm{int}}=K_{\mathrm{ext}}.
$$

A basic example illustrating this situation is the standard sphere in $\r^3$ of radius $r$, whose principal curvatures are constant and equal to $1/r$.

\section*{Acknowledgments}
This study was supported by Scientific and Technological Research Council of Turkey (TUBITAK) under the Grant Number (123F451). The authors thank to TUBITAK for their supports.

\section*{Declarations}

\begin{itemize}
\item Conflict of interest: The authors declare no conflict of interest.

\item Ethics approval: Not applicable.

\item Availability of data and materials: Not applicable.

\item Code availability: Not applicable.

\item Authors’ contributions: Muhittin Evren Aydın formulated the problem, performed the main analysis and designed the structure of the manuscript. Esra Dilmen and Büşra Karakaya contributed to solving the problem and to writing the manuscript. All authors read and approved the final version of the paper.
\end{itemize}


 \end{document}